\documentclass{amsart}
\usepackage{amsmath, amsthm}
\usepackage{amsmath,amscd}
\usepackage{amsfonts}
\usepackage{amssymb, latexsym}
\usepackage{xcolor}

\theoremstyle{plain}
\newtheorem{thm}{Theorem}[section]
\newtheorem{cor}[thm]{Corollary}
\newtheorem{prop}[thm]{Proposition}
\newtheorem{lem}[thm]{Lemma}

\theoremstyle{definition}
\newtheorem{defn}{Definition}[section]

\newcommand{\ga}{\alpha}
\newcommand{\gb}{\beta}
\newcommand{\gfi}{\varphi}
\renewcommand{\gg}{\gamma}
\newcommand{\gl}{\lambda}
\newcommand{\gp}{\pi}
\newcommand{\gm}{\mu}

\newcommand{\gt}{\tau}

\newcommand{\gs}{\sigma}
\newcommand{\gz}{\zeta}

\newcommand{\N}{\mathbb{N}}
\newcommand{\Q}{\mathbb{Q}}
\newcommand{\F}{\mathcal{F}}
\newcommand{\FT}{\mathcal{FT}}
\newcommand{\FP}{\mathcal{FT_{\pi}}}
\newcommand{\C}{\mathcal{C}}
\newcommand{\Qx}{\mathbb{Q}[x]}
\newcommand{\nilR}{nilpotent $R$-powered group}
\newcommand{\nilQx}{nilpotent $\mathbb{Q}[x]$-powered group}

\copyrightinfo{2006} {American Mathematical Society}

\begin{document}

\title{Separability Properties of Nilpotent $\Qx$-Powered Groups}

\author{Stephen Majewicz}
\address{Department of Mathematics\\
         CUNY-Kingsborough Community College\\
         Brooklyn, New York 11235}
\email{smajewicz@kbcc.cuny.edu}

\author{Marcos Zyman}
\address{Department of Mathematics\\
         CUNY-Borough of Manhattan Community College\\
         New York, New York 10007}
\email{mzyman@bmcc.cuny.edu}


\date{March 19, 2019}

\dedicatory{This paper is dedicated to Ben Fine and Anthony Gaglione on the occasion of their birthdays; and to Gilbert Baumslag (in memoriam)}


\begin{abstract}
In this paper we study conjugacy and subgroup separability properties in the class of \nilQx s. Many of the techniques used to study these properties in the context of ordinary nilpotent groups carry over naturally to this more general class. Among other results, we offer a generalization of a theorem due to G. Baumslag. The generalized version states that if $G$ is a finitely $\Qx$-generated $\Qx$-torsion-free \nilQx \ and $H$ is a $\Qx$-isolated subgroup of $G,$ then for any prime $\gp \in \Qx$, $\bigcap_{i = 1}^{\infty} G^{{\gp}^{i}}H = H.$
\end{abstract}

\maketitle

\section{Introduction}

Let $\C$ be a class of groups. Among the types of separability properties of groups, there are:
\begin{itemize}
\item \underline{Conjugacy $\C$-Separability} : A group $G$ is \emph{conjugacy $\C$-separable} if whenever $g$ and $h$ are not conjugate elements in $G,$ there exists $K \in \C$ and a homomorphism $\gfi$ from $G$ onto $K$ such that $\gfi(g)$ and $\gfi(h)$ are not conjugate elements in $K.$

\vspace{.1in}

\item \underline{Residually $\C$} : A group $G$ is \emph{residually $\C$} if for every non-identity element $g$ in $G,$ there exists $K \in \C$ and a homomorphism $\gfi$ from $G$ onto $K$ such that $\gfi(g)$ is not the identity.

  {Equivalently, for some non-empty index set $\Lambda,$ there exists a family $\{N_{\gl} \ | \ \gl \in \Lambda \}$ of normal subgroups of $G$ such that $G/N_{\gl} \in \C$ for all $\gl \in \Lambda$ and $\cap_{\gl \in \Lambda}N_{\gl} = 1.$}

\vspace{.1in}

\item \underline{Subgroup $\C$-Separability} : A subgroup $H$ of a group $G$ is called \emph{$\C$-separable} in $G$ if, for every element $g \in G \setminus H,$ there is a homomorphism $\gfi$ of $G$ onto $K \in \C$ such that $\gfi(g) \notin \gfi(H).$ The group $G$ is said to be \emph{separated in a group from $\C$}, and $K$ is said to \emph{separate $g$ and $H$}.
\end{itemize}

It is clear that $G$ is residually $\C$ if and only if the identity subgroup of $G$ is $\C$-separable. Furthermore, every conjugacy $\C$-separable group is residually $\C.$

Let $\F$ and $\F _{p}$ denote the classes of finite groups and finite $p$-groups ($p$ is a prime) respectively. In the literature, there are several results pertaining to these properties for nilpotent groups. For instance,

\vspace{.1in}

\noindent (i) \ In \cite{blackburn-1965}, N. Blackburn proved that every finitely generated nilpotent group is conjugacy $\F$-separable.

\vspace{.1in}

\noindent (ii) \ In \cite{ivanova-2004}, E. A. Ivanova proved that if $G$ is a finitely generated torsion-free nilpotent group which is conjugacy $\F _{p}$-separable for some prime $p,$ then $G$ must be abelian.

\vspace{.1in}

\noindent (iii) \ Let $G$ be a finitely generated torsion-free nilpotent group, and let $H$ be an isolated subgroup of $G.$ In \cite{baumslag-1971}, G. Baumslag proved that for any prime $p,$
\[
\bigcap_{i = 1}^{\infty} G^{p^{i}}H = H.
\]
Equivalently, $H$ is $\F _{p}$-separable in $G$ for any given prime $p.$

\vspace{.1in}

{If $H$ is the trivial subgroup in (iii), then we get a classical theorem of K. W. Gruenberg \cite{gruenberg-1957} which states that finitely generated torsion-free nilpotent groups are residually $\F _{p}$ for any prime $p.$}

\vspace{.1in}

\noindent (iv) \  In \cite{allenby_and_gregorac-1972}, R. B. J. T. Allenby and R. J. Gregorac prove the following: Let $G$ be a finitely generated nilpotent group and $H \leq G.$ Suppose there is an element $g \in G \setminus H$ such that $g^{n} \notin H$ for any positive integer $n.$ If $p$ is any prime, then $g$ and $H$ can be separated in a finite $p$-group.

\vspace{.25in}

Our goal in this paper is to prove results similar to those mentioned above for nilpotent $\Qx$-powered groups. The ring $\Qx$ is an example of a \emph{binomial ring}, an integral domain of characteristic zero with unity such that for any $r \in \Qx$ and positive integer $k,$
\[
\binom{r}{k} = \frac{r(r - 1)\cdots(r - k + 1)}{k!} \in \Qx.
\]
In this paper, $R$ will always be a binomial ring.

\begin{defn}\cite{hall-1969}
Let $G$ be a (locally) nilpotent group which comes equipped with an action
by $R,$
\[
G \times R \rightarrow G \hbox{ \ defined by \ } (g, \ \ga) \mapsto g^{\ga},
\]
such that for all $g \in G$ and for all $\ga \in R$, the element $g^{\ga} \in G$ is uniquely determined. Then $G$
is called a \emph{nilpotent $R$-powered group} if the following axioms hold:
\begin{itemize}
\item $g^{1} = g, \, g^{\ga}g^{\gb} = g^{\ga + \gb},$ and $(g^{\ga})^{\gb} = g^{\ga\gb}$ for all $g \in G$ and $\ga, \gb \in
R.$

\vspace{.1in}

\item $(h^{-1}gh)^{\ga} = h^{-1}g^{\ga}h$ for all $g, \ h \in G$ and for all $\ga \in R.$

\vspace{.1in}

\item (Hall-Petresco Axiom) If $\{g_{1}, \, \ldots, \, g_{n}\} \subset G$ and $\ga \in R,$ then
\[
g_{1}^{\ga} \cdots g_{n}^{\ga} = \tau_{1}(\bar{g})^{\ga}\gt_{2}(\bar{g})^{\binom{\ga}{2}} \cdots \gt_{k - 1}(\bar{g})^{\binom{\ga}{k-1}}\gt_{k}(\bar{g})^{\binom{\ga}{k}},
\]
where $k$ is the nilpotency class of the group generated by $\{g_{1}, \, g_{2}, \, \ldots, \, g_{n}\},$ $\bar{g} = (g_{1},
\, \ldots, \, g_{n}),$ and $\tau_{i}(\bar{g})$ is the $i$th \emph{Hall-Petresco word}.
\end{itemize}
\end{defn}

The Hall-Petresco words are obtained by letting $\ga = 1, \ 2, \ $ and so on. For instance,
\[
\tau_{1}(\bar{g}) = g_{1} \cdots g_{n} \hbox{ and } \gt_{2}\left(\overline{g}\right) = \left(g_{1} \cdots g_{n}\right)^{-2}g_{1}^{2} \cdots g_{n}^{2}.
\]

$R$-modules and $R$-completions of finitely generated torsion-free nilpotent groups with respect to a Mal'cev basis are examples of \nilR s \cite{hall-1969}. Categorical notions such as $R$-subgroup, $R$-homomorphism, etc. are defined in the obvious way \cite{hall-1969}.

Various theorems on residual properties of \nilR s have already been proven in some papers by the authors (\cite{majewicz-2006}, \cite{majewicz_and_zyman-2012}, \cite{majewicz_and_zyman-2012(2)}). A result on the conjugacy separability of \nilQx s can also be found in \cite{majewicz-2006}.

The definitions of separability carry over to \nilR s in a natural way. We merely state them here without giving additional background definitions. All necessary definitions, notations, and results on \nilR s will be provided in the next section.

\begin{defn}\label{ConjSep}
Let $\C$ denote a class of \nilR s. A \nilR \ $G$ is \emph{conjugacy $\C$-separable} if whenever $g$ and $h$ are not conjugate elements in $G,$ there exists $K \in \C$ and an $R$-homomorphism $\gfi$ of $G$ onto $K$ such that $\gfi(g)$ and $\gfi(h)$ are not conjugate in $K.$
\end{defn}

From this point on, $\FT$ and $\FP$ will denote the classes of \nilQx s of finite type and finite $\gp$-type for a prime $\gp \in \Qx$ respectively.

\begin{thm}\label{t:Thm1}
Let $G$ be a finitely $\Qx$-generated \nilQx. Two elements of $G$ are conjugate if and only if they are conjugate in every factor $\Qx$-group of $G$ of finite type.
\end{thm}

This establishes that finitely $\Qx$-generated \nilQx s are conjugacy $\FT$-separable, thus proving an analogue of (i). Our proof mimics the one given by G. Baumslag for Theorem 1.3 of \cite{baumslag-1971}.

\medskip

Our next theorem is similar to (ii).

\begin{thm}\label{t:Thm2}
Let $G$ be a finitely $\Qx$-generated $\Qx$-torsion-free \nilQx. If $G$ is conjugacy $\FP$-separable for some prime $\pi \in \Qx,$ then $G$ is an abelian $\Qx$-group.
\end{thm}

The proof mimics the one given by E. A. Ivanova in \cite{ivanova-2004}.

\begin{defn}\label{d:R-SubgroupSeparability}
Let $G$ be a \nilR, and let $\C$ be a class of \nilR s. An $R$-subgroup $H$ of $G$ is called \emph{$\C$-separable} in $G$ if, for every element $g \in G \setminus H,$ there is an $R$-homomorphism $\gfi$ of $G$ onto $K \in \C$ such that $\gfi(g) \notin \gfi(H).$ The group $G$ is said to be \emph{separated in a group from $\C$}, and $K$ is said to \emph{separate $g$ and $H$}.
\end{defn}

We prove the following analogue of (iii):

\begin{thm}\label{t:Thm3}
Let $G$ be a finitely $\Qx$-generated $\Qx$-torsion-free \nilQx, and let $H$ be a $\Qx$-isolated subgroup of $G.$ For any prime $\gp \in \Qx$,
\[
\bigcap_{i = 1}^{\infty} G^{{\gp}^{i}}H = H.
\]
Equivalently, $H$ is $\FP$-separable for any given prime $\gp \in \Qx.$
\end{thm}

Another result on $\FP$-separability is:

\begin{thm}\label{t:Thm4}
Let $G$ be a finitely $\Qx$-generated \nilQx \ and $H \leq_{\Qx} G.$ Suppose there is an element $g \in G \setminus H$ such that $g^{\gl} \notin H$ for any $0 \neq \gl \in \Qx.$ If $\gp$ is any prime in $\Qx,$ then $g$ and $H$ can be separated in a \nilQx \ of finite $\gp$-type.
\end{thm}

The proof is similar to the one given for Theorem 1 (a) in \cite{allenby_and_gregorac-1972}.

\section{Preliminaries}

In this section, we discuss some terminology and results on \nilR s which will be used in this paper. We remind the reader that $R$ will always be a binomial ring. We will usually denote a group identity element, a multiplicative identity ring element, and the trivial subgroup of a group by $1.$

\vspace{.15in}

 Let $G$ is a \nilR. If $[g, \ h] = 1$ for $g, \ h \in G$, the Hall-Petresco axiom implies that $(gh)^{\ga} = g^{\ga}h^{\ga}$ for any $\ga \in R.$
In particular, every abelian $R$-group is an $R$-module.

\begin{prop} \label{p:CommuteCenter}
Suppose $G$ is a \nilR. Let $g, \ h \in G$ such that $[g, \ h] \in Z(G).$ Then $[g, \ h]^{\ga} = [g, \ h^{\ga}]$ for all $\ga \in R$.
\end{prop}

\begin{proof}
Since $[g, \ h] \in Z(G),$ we have
\begin{eqnarray*}
[g^{-1}h^{-1}g, \ h] & = & g^{-1}hgh^{-1}[g, \ h]\\
                     & = & g^{-1}hg[g, \ h]h^{-1}\\
                     & = & 1.
\end{eqnarray*}
By the previous paragraph, $(g^{-1}h^{-1}gh)^{\ga} = (g^{-1}h^{-1}g)^{\ga}h^{\ga}$ for any $\ga \in R.$ Since $(g^{-1}h^{-1}g)^{\ga} = g^{-1}h^{-\ga}g,$ the result follows.
\end{proof}

Let $G$ be a \nilR. If $H$ is an $R$-subgroup or a normal $R$-subgroup of $G,$ then we write $H \leq_{R} G \hbox{ and } H \unlhd_{R} G$ respectively. We say that $H$ is \emph{$R$-generated by $X = \{x_{1}, \, \ldots, \, x_{j}\} \subseteq G$} if
\[
H = \bigcap_{X \subset H_{i}\leq_{R}G} \{H_{i}\}.
\]
In this case, we write $H = gp_{R}(x_{1}, \, \ldots, \, x_{j}).$

If $G$ is a \nilR \ and $N \unlhd_{R} G,$ then the $R$-action on $G$
induces an $R$-action on $G/N,$
\[
(gN)^{\ga} = g^{\ga}N \hbox{ for all } gN \in G/N \hbox{ and } \ga \in R,
\]
which turns $G/N$ into a \nilR.

It is not hard to show using the Hall-Petresco Axiom that if $G$ is a \nilR \ with $H \leq_{R} G$ and $N \unlhd_{R} G,$ then $HN \leq_{R} G$ and $HN = gp_{R}(H, \ N).$

The usual isomorphism theorems carry over to nilpotent $R$-powered groups in the obvious way. If $G_{1}$ and $G_{2}$ are \nilR s, then the kernel of an $R$-homomorphism $\gfi : G_{1} \rightarrow G_{2},$ abbreviated as $ker \ \gfi,$ is a normal $R$-subgroup of $G_{1}.$ If $G_{1}$ and $G_{2}$ happen to be $R$-isomorphic, then we write $G_{1} \cong_{R} G_{2}.$

The upper and lower central subgroups of a \nilR \ $G$ are $R$-subgroups of $G.$ We denote them by $\gz_{i}G$ and $\gg_{i}G$ respectively. The center of $G$ will be written as $Z(G).$

\begin{thm}\cite{kargapolov and et al-1969}\label{t:MAX}
If $R$ is noetherian, then every $R$-subgroup of a finitely $R$-generated \nilR \ is finitely $R$-generated.
\end{thm}

\begin{defn}\cite{warfield}\label{d:RTorsion}
If $G$ is a \nilR, then $g \in G$ is an \emph{$R$-torsion element} if $g^{\ga} = 1$ for some $0 \neq \ga \in R.$
If every element of $G$ is $R$-torsion, then $G$ is called an \emph{$R$-torsion group}. We say that $G$ is \emph{$R$-torsion-free} if the only $R$-torsion element of $G$ is $1.$
\end{defn}

\begin{defn}\cite{majewicz-2006}\label{d:finite0-type}
Let $\gp \in R$ be a prime. A finitely $R$-generated \nilR \ $G$ is of
\begin{enumerate}
\item \emph{finite type} if it is an $R$-torsion group.
\item \emph{finite $\gp$-type} if for each $g \in G,$ there exists a positive integer $k$ such that $g^{\gp^{k}} = 1.$
\end{enumerate}
\end{defn}

The next theorem and lemma are proven by the authors in \cite{majewicz_and_zyman-2009}.

\begin{thm}\label{t:ShortExactSequence}
Let $R$ be noetherian and consider the short exact sequence
\[
1 \rightarrow H \rightarrow G \rightarrow G/H \rightarrow 1
\]
in the category of \nilR s. Then $G$ is of finite type if and only if $H$ and $G/H$ are both of finite type.
\end{thm}

\begin{lem}\label{l:Lemma1}
Let $R$ be a PID. If $G$ is a finitely $R$-generated \nilR \ and $Z(G)$ is of finite type, then $G$ is of finite type.
\end{lem}

The next lemma will be needed later.

\begin{lem}\label{l:Lemma2}
Let $R$ be a PID. If $G$ is a finitely $R$-generated \nilR \ and is not of finite type, then $Z(G)$ has an element which is not $R$-torsion.
\end{lem}

\begin{proof}
Assume, on the contrary, that $Z(G)$ is an $R$-torsion group. Since $G$ is finitely $R$-generated, so is $Z(G)$ by Theorem~\ref{t:MAX}. Thus, $Z(G)$ is of finite type. By Lemma~\ref{l:Lemma1}, $G$ is also of finite type.
\end{proof}

If $G$ is a \nilR \ and $\gb \in R,$ then we define
\[
G^{\gb} = gp_{R}(g^{\gb} \ | \ g \in G).
\]
It is clear that $G^{\gb} \unlhd_{R} G.$

\begin{thm}\cite{majewicz_and_zyman-2010}\label{t:ProductOfPowers}
If $R$ contains $\Q$ and $G$ is a \nilR, then every element of $G^{\gb}$ is a $\gb$th power of an element of $G$ for any $\gb \in R.$
\end{thm}

The notion of the exponent of a group has been generalized for the category of \nilR s and will be used in the proof of Theorem 1.4.

\begin{defn}\cite{majewicz_and_zyman-2009}\label{d:ExponentOfNilR}
Let $R$ be a PID, and let $G$ be a \nilR \, of finite type. We say that $G$ has \emph{exponent} $\pi_{1} \pi_{2} \cdots \pi_{n},$ where the $\pi_{i}$'s are primes (not necessarily distinct) in $R,$ if
\[
G^{\pi_{1} \pi_{2} \cdots \pi_{n}} = 1 \hbox{ but } G^{\pi_{1} \pi_{2} \cdots \widehat{\pi_{i}} \cdots \pi_{n}} \neq 1
\hbox{ for any } 1 \leq i \leq n.
\]
\end{defn}

Here, $\pi_{1} \pi_{2} \cdots \widehat{\pi_{i}} \cdots \pi_{n}$ stands for the omission of $\pi_{i}$ in the product $\pi_{1} \pi_{2} \cdots \pi_{n}.$

\vspace{.1in}

The $R$-isolator of a \nilR \ plays an important role in what follows.

\begin{defn}\label{d:R-isolated}
Let $G$ be a \nilR.\\
\noindent (i) \ An $R$-subgroup $H$ of $G$ is \emph{$R$-isolated} in $G$ if $g \in G$ and $g^{\gl} \in H$ for some $0 \neq \gl \in R$ imply $g \in H.$

\vspace{.1in}

\noindent (ii) \ The \emph{$R$-isolator} of $S \subseteq G,$ denoted by $I(S, \ G),$ is the intersection of all $R$-isolated subgroups of $G$ containing $S.$
\end{defn}

Clearly, $I(S, \ G)$ is an $R$-isolated subgroup of $G.$ The next theorem describes the $R$-isolator of an $R$-subgroup and can be found in \cite{warfield}.

\begin{thm}\cite{warfield}\label{t:R-Isolator}
If $G$ is a \nilR \ and $H \leq_{R} G,$ then
\[
I(H, \ G) = \{g \in G \ | \ g^{\gl} \in H \hbox{ for some } 0 \neq \gl \in R \}.
\]
\end{thm}

Note that $I(1, \ G)$ is the $R$-subgroup of $G$ consisting of the $R$-torsion elements of $G.$ This is referred to as the \emph{$R$-torsion subgroup} of $G$ and denoted by $\gt(G).$ It is clear that $\gt(G) \unlhd_{R} G$ and $G/\gt(G)$ is $R$-torsion-free.

\medskip

We record a simple lemma which will be used in what follows.

\begin{lem}\label{l:R-IsolatedAndR-TFree}
Let $G$ be a \nilR \ and $H \unlhd_{R} G.$ Then $G/H$ is $R$-torsion-free if and only if $H$ is $R$-isolated in $G.$
\end{lem}

The theory of extraction of roots appears in several papers by the authors (see \cite{majewicz_and_zyman-2009}, \cite{majewicz_and_zyman-2010},  \cite{majewicz_and_zyman-2012}, and \cite{majewicz_and_zyman-2012(2)}).

\begin{defn}\label{d:Root}
If $G$ is a \nilR \ and $h^{\ga} = g$ for some $g, \ h \in G$ and $\ga \in R,$ then $h$ is called an \emph{$\ga$th root} of $g.$
\end{defn}

\begin{thm}\cite{warfield}\label{t:UniqueRoots}
Let $G$ be a \nilR. Then $G$ is $R$-torsion-free if and only if every $1 \neq g \in G$ has at most one $\ga$th root for every $0 \neq \ga \in R;$ that is, if $g, \ h \in G$ and $g^{\ga} = h^{\ga}$ for some $0 \neq \ga \in R,$ then $g = h.$
\end{thm}

\begin{cor}\label{c:UniqueRoots}
If $G$ is an $R$-torsion-free \nilR, then $Z(G)$ is $R$-isolated.
\end{cor}

\begin{proof}
If $g^{\ga} \in Z(G)$ for some $g \in G$ and $\ga \in R,$ then $h^{-1}g^{\ga}h = g^{\ga}$ for any $h \in G.$ Thus, $(h^{-1}gh)^{\ga} = g^{\ga}.$ By Theorem~\ref{t:UniqueRoots}, $h^{-1}gh = g,$ and thus, $g \in Z(G).$
\end{proof}

Another result that we will rely on is:

\begin{lem}\label{l:NoInfiniteRoots}
Let $G$ be a non-trivial finitely $\Qx$-generated $\Qx$-torsion-free \nilQx. For any $1 \neq g \in G$ and prime $\gp \in \Qx,$ there exists $n \in \N$ such that $g$ has no $\gp^{n}$th root in $G.$
\end{lem}

\begin{proof}
Assume, on the contrary, that $g$ has $\gp^{n}$th roots in $G$ for every $n \in \N.$ By Theorem~\ref{t:UniqueRoots}, these $\gp^{n}$th roots must be unique. Thus, we let
\[
H = gp_{\Qx}(g_{1}, \ g_{2}, \ \ldots ), \hbox{ where } g = g_{1}^{\gp} \hbox{ and } g_{j} = g_{j + 1}^{\gp} \hbox{ for all } j \in \N.
\]
By Theorem~\ref{t:MAX}, $H$ is finitely $\Qx$-generated. Consequently, there exists $k \in \N$ such that $H =  gp_{\Qx}(g_{1}, \ g_{2}, \ \ldots, \ g_{k}).$ By construction of the $g_{j}$'s, this means that $H =  gp_{\Qx}(g_{k}).$

Now, $H$ should contain all $\gp^{n}$th roots of $g.$ However, this is not the case since $g_{k},$ the $\gp^{k}$th root of $g,$ has no $\gp$th root in $H.$ For if $g_{k}$ had a $\gp$th root in $H,$ then there would exist $\gm \in \Qx$ such that $(g_{k}^{\gm})^{\gp} = g_{k};$ that is, $g_{k}^{\gm \gp - 1} = 1.$ Since $H$ is $\Qx$-torsion-free and $g_{k} \neq 1,$ $\gm \gp - 1 = 0.$ It follows that $\gp$ must be a unit. However, $\gp$ is a prime by hypothesis, a contradiction.
\end{proof}

Next we discuss some results on the $R$-torsion-free rank of a finitely $R$-generated \nilR, where $R$ is noetherian (see \cite{warfield}). First we recall the notion of torsion-free rank for a module.

\begin{defn}\label{d:RankOfModule}
Let $D$ be an integral domain with fraction field $K,$ and let $M$ be a $D$-module. The \emph{torsion-free rank} of $M$ is the $K$-dimension of the $K$-vector space $K \otimes_{D} M,$ denoted by $dim_{K}(K \otimes_{D} M).$
\end{defn}

If $M$ is a finitely generated $D$-module, then $K \otimes_{D} M$ is finite dimensional as a $K$-vector space and $dim_{K}(K \otimes_{D} M)$ is the maximal number of $D$-linearly independent elements in $M,$ i.e. the maximal rank of a free $D$-submodule of $M$  (see Theorem 6.11 in \cite{conrad2} for a proof). In the literature, torsion-free rank is usually referred to simply as rank. In particular, if $D$ is a PID and
\[
M = \underbrace{D \oplus \cdots \oplus D}_{m \, summands} \oplus T,
\]
where $T$ is the torsion part of $M,$ then the (torsion-free) rank of $M$ is $m.$

\begin{defn}\label{d:R-TorsionFreeRank}
Let $G$ be a finitely $R$-generated \nilR, where $R$ is noetherian, and let $K$ be the fraction field of $R.$ If
\[
G = G_{1} \unrhd_{R} G_{2} \unrhd_{R} \cdots \unrhd_{R} G_{n + 1} = 1
\]
is a subnormal $R$-series, where $G_{i}/G_{i + 1}$ is an $R$-module for $i = 1, \ 2, \ \ldots, \ n,$ then the \emph{$R$-torsion-free rank} of $G$ is
\[
\sum_{i = 1}^{n} dim_{K}(K \otimes_{R} G_{i}/G_{i + 1}).
\]
\end{defn}

\vspace{.1in}

It follows from Theorem~\ref{t:MAX} that each quotient $G_{i}/G_{i + 1}$ is a finitely generated $R$-module of finite rank. Thus, the above sum is well-defined and $G$ has finite $R$-torsion-free rank. An application of the Jordan-H\"{o}lder theorem shows that the $R$-torsion-free rank of a \nilR \ is independent of the subnormal $R$-series chosen \cite{warfield}.

\medskip

The next theorem and corollary are extensions of known facts about the rank of a finitely generated $R$-module over a PID.

\begin{thm}\label{t:R-TorsionFreeRank}
Let $R$ be a PID, and suppose $G$ is a finitely $R$-generated \nilR \ with $R$-torsion-free rank $n.$

\vspace{.1in}

\noindent (i) \ If $H \leq_{R} G,$ then the $R$-torsion-free rank of $H$ is at most $n.$

\vspace{.1in}

\noindent (ii) \ If $N \unlhd_{R} G$ and $N$ has $R$-torsion-free rank $m,$  then $G/N$ has $R$-torsion-free rank $n - m.$
\end{thm}

\begin{proof}
(i) \ Suppose $G$ has nilpotency class $c,$ and let
\[
G = \gg_{1}G \geq \gg_{2}G \geq \cdots \geq \gg_{c + 1}G = 1
\]
be the lower central series of $G.$ Assume that $\gg_{i}G/\gg_{i + 1}G$ has $R$-torsion-free rank $n_{i}$ for $i = 1, \ 2, \ \ldots, \ c.$ Since $G$ has $R$-torsion-free rank $n,$ we have $n_{1} + \cdots + n_{c} = n.$ Put $H_{i} = H \cap \gg_{i}G$ for $i = 1, \ 2, \ \ldots, \ c.$ Then
\[
H = H_{1} \geq_{R} H_{2} \geq_{R} \cdots \geq_{R} H_{c} \geq_{R} H_{c + 1} = 1
\]
is a central $R$-series for $H.$ Now,
\[
\frac{H_{i}}{H_{i + 1}} = \frac{H \cap \gg_{i}G}{H \cap \gg_{i + 1}G} = \frac{H \cap \gg_{i}G}{(H \cap \gg_{i}G) \cap \gg_{i + 1}G} \cong_{R} \frac{\gg_{i + 1}G (H \cap \gg_{i}G)}{\gg_{i + 1}G}
\]
for $i = 1, \ 2, \ \ldots, \ c.$ Therefore, $H_{i}/H_{i + 1}$ is $R$-isomorphic to an $R$-subgroup of $\gg_{i}G/\gg_{i + 1}G.$ Since $\gg_{i}G/\gg_{i + 1}G$ is an $R$-module, the $R$-torsion-free rank of $H_{i}/H_{i + 1}$ is at most $n_{i}.$ Hence, the $R$-torsion-free rank of $H$ is at most $n.$

\vspace{.1in}

\noindent (ii) \ Let
\[
1 = N_{0} \unlhd_{R} N_{1} \unlhd_{R} \cdots \unlhd_{R} N_{k} = N
\]
and
\[
1 = G_{0}/N \unlhd_{R} G_{1}/N \unlhd_{R} \cdots \unlhd_{R} G_{l}/N = G/N
\]
be subnormal $R$-series of $N$ and $G/N$ respectively. Clearly, $G$ has a subnormal $R$-series
\[
1 = N_{0} \unlhd_{R} N_{1} \unlhd_{R} \cdots \unlhd_{R} N_{k} \unlhd_{R} G_{1} \unlhd_{R} \cdots \unlhd_{R} G_{l} = G.
\]
Since $(G_{i + 1}/N)/(G_{i}/N) \cong_{R} G_{i + 1}/G_{i}$ for $i = 0, \ 1, \ \ldots, \ l - 1,$ the $R$-torsion-free rank of $G$ is the sum of the $R$-torsion-free ranks of $N$ and $G/N.$
\end{proof}

The next corollary follows immediately from Theorem~\ref{t:R-TorsionFreeRank} (ii).

\begin{cor}\label{c:R-TorsionFreeRank}
Let $R$ be a PID, and let $G$ be a finitely $R$-generated $R$-torsion-free \nilR. If $N \unlhd_{R} G,$ then $G$ and $N$ have the same $R$-torsion-free rank if and only if $G/N$ is an $R$-torsion group.
\end{cor}

\begin{cor} \label{c:Abelian}
Let $R$ be a PID, and let $G$ be a finitely $R$-generated $R$-torsion-free \nilR. If the $R$-torsion-free rank of $G$ is 1, then $G$ is an $R$-module.
\end{cor}

\begin{proof}
It suffices to show that $G$ is abelian. Since $Z(G)$ is $R$-torsion-free, its $R$-torsion-free rank equals some $m > 0.$ By Theorem~\ref{t:R-TorsionFreeRank} (i), $m$ must equal $1.$ Hence, $G/Z(G)$ is an $R$-torsion group by Corollary~\ref{c:R-TorsionFreeRank}. Now, $Z(G)$ is $R$-isolated by Corollary~\ref{c:UniqueRoots}. By Lemma~\ref{l:R-IsolatedAndR-TFree}, $G/Z(G)$ must be trivial. Thus, $G$ is abelian.
\end{proof}

\section{Proofs of Theorems}

We use the following notation: If $g$ and $h$ are (are not) conjugate group elements, then we write $g \sim h$ ($g \nsim h$).

\medskip

\noindent \underline{Proof of Theorem~\ref{t:Thm1}}: The proof is by induction on the $\Qx$-torsion-free rank $r$ of $G.$ If $r = 0,$ then $G$ is of finite type and there is nothing to prove.

Suppose that $r > 0.$ Let $g$ and $h$ be elements of $G$ which are conjugate in every factor $\Qx$-group of finite type. We claim that $g \sim h.$ Assume, on the contrary, that $g \nsim h.$ Since $r > 0,$ $G$ is not of finite type. By Lemma~\ref{l:Lemma2}, there exists $a \in Z(G)$ that is not a $\Qx$-torsion element. Let
\[
S = \{1, \ f \in \Qx \ | \ f \hbox{ is monic}\},
\]
and for each $\gs \in S,$ define
\[
H_{\gs} = gp_{\Qx}\left(a^{\gs} \right).
\]
For each $\gs \in S$, observe that $H_{\gs} \unlhd_{\Q[x]} G.$ By Theorem~\ref{t:R-TorsionFreeRank}, the $\Qx$-torsion-free rank of each $G/H_{\gs}$ is less than $r.$

Suppose that $gH_{\gs} \nsim hH_{\gs}$ in $G/H_{\gs}$ for some $\gs \in S.$ By induction and a $\Qx$-Isomorphism Theorem, there exists a normal $\Qx$-subgroup $N_{\gs}/H_{\gs}$ of $G/H_{\gs}$ such that
\[
\left(G/H_{\gs} \right) / \left( N_{\gs} / H_{\gs}\right) \cong_{\Qx} G/N_{\gs}
\]
is of finite type and the images of $g$ and $h$ under the natural $\Qx$-homomorphism from $G$ onto $G/N_{\gs}$ are not conjugate in $G/N_{\gs}.$ This contradicts the assumption that the images of $g$ and $h$ are conjugate in every factor $\Qx$-group of $G$ of finite type. Therefore, $gH_{\gs} \sim hH_{\gs}$ for all $\gs \in S.$ In particular, $gH_{1} \sim hH_{1}$ in $G/H_{1}.$ Thus, since $a \in Z(G),$ we can find a non-zero element $t \in \Qx$ such that $h \sim ga^{t}.$ Since $gH_{\gs} \sim hH_{\gs},$ we have
\[
gH_{\gs} \sim ga^{t}H_{\gs} \hbox{ \ for each \ } \gs \in S.
\]
Thus, for each $\gs \in S,$ we can find $y_{\gs} \in G$ and $u_{\gs} \in \Qx$ such that
\begin{equation}\label{e:ConjugacyProblem1}
y_{\gs}^{-1}gy_{\gs} = ga^{t} \left(a^{\gs}\right)^{u_{\gs}}.
\end{equation}
Let $Y$ be the set of those $y_{\gs} \in G$ arising in (\ref{e:ConjugacyProblem1}), and put $L = gp_{\Qx}(g, \ a, \ Y).$ Clearly, $\left[ y, \ g \right] \in gp_{\Qx}(a)$ for all $y \in Y$ by (\ref{e:ConjugacyProblem1}). Moreover, $gp_{\Qx}(a) \leq_{\Qx} Z(L)$ since $a \in Z(G).$ It follows from the commutator calculus and the axioms that $[l, \ g] \in gp_{\Qx}(a)$ for all $l \in L.$ Thus, we obtain a $\Qx$-homomorphism
\[
\gfi : L \rightarrow gp_{\Qx}(a) \hbox{ \ defined by \ } \gfi(l) = [l, \ g].
\]
Since $\Qx$ is a PID, $L/ ker \ \gfi$ is $\Qx$-cyclic. Thus, there exists an element $b \in L \setminus ker \ \gfi$ and $\ga \in \Qx$ such that
\begin{equation}\label{e:NewFormL}
L = gp_{\Qx}(ker \ \gfi, \ b) \hbox{ \ and \ } [b, \ g] = a^{\ga}.
\end{equation}
Note that $\ga \neq 0$ because $b \notin Ker \ \gfi.$ In fact, $b$ and $\ga$ can be chosen so that $\ga \in S.$

Next, we find all of the conjugates of $g$ in $L.$ By (\ref{e:NewFormL}), such a conjugate has the form $(b^{\gm}k)^{-1}g(b^{\gm}k),$ where $k \in ker \ \gfi$ and $\gm \in \Qx.$ Since $[g, \ b] = a^{-\ga} \in Z(L)$ and $ker \ \gfi$ is the centralizer of $g$ in $L,$
\begin{eqnarray*}
(b^{\gm}k)^{-1}g(b^{\gm}k) & = & k^{-1}b^{-\gm}gb^{\gm}k = k^{-1}g[g, \ b^{\gm}]k\\
                           & = & k^{-1}g[g, \ b]^{\gm}k = k^{-1}gk[g, \ b]^{\gm}\\
                           & = & ga^{-\ga\gm}
\end{eqnarray*}
by Proposition~\ref{p:CommuteCenter}. It follows that the set of conjugates of $g$ in $L$ is
\begin{equation}\label{e:ConjList}
C = \{ga^{\ga\gm} \ | \ \gm \in \Qx\}.
\end{equation}
Now, $g \nsim ga^{t}$ in $L$ because $g \nsim ga^{t}$ in $G.$ Thus $ga^{t} \notin C,$ and so, $\ga\gm \neq t$ for all $\gm \in \Qx.$ On the other hand, (\ref{e:ConjugacyProblem1}) yields
\[
y_{\ga}^{-1}gy_{\ga} = ga^{t}\left(a^{\ga}\right)^{u_{\ga}}
\]
because $\ga \in S.$ This means that $ga^{t}\left(a^{\ga}\right)^{u_{\ga}} \in C,$ so $t + \ga u_{\ga} = \ga \gl$ for some $\gl \in \Qx.$ Hence, $t = \ga(\gl - u_{\ga}),$ a contradiction since $\gl - u_{\ga} \in \Qx$. Theorem~\ref{t:Thm1} is proven.

\vspace{.25in}

\noindent \underline{Proof of Theorem~\ref{t:Thm2}}: Suppose that $G$ has nilpotency class $c,$ and assume that $G$ is not abelian. In this case, $c > 1.$ Thus, there exist elements $z_{2} \in \gz_{2}G \setminus Z(G)$ and $h \in G$ such that $[z_{2}, \ h] = z_{1}$ for some $1 \neq z_{1} \in Z(G).$ Let $\ga$ be a prime in $\Qx$ different from $\gp.$ Since $\Qx$ is noetherian and $G$ is finitely $\Qx$-generated, $Z(G)$ is also finitely $\Qx$-generated by Theorem~\ref{t:MAX}. Hence, $Z(G)$ is a finitely $\Qx$-generated $\Qx$-torsion-free abelian $\Qx$-group. By Lemma~\ref{l:NoInfiniteRoots}, there exists $n \in \N$ such that $z_{1}$ has no $\ga^{n}$th root in $Z(G).$ We claim that $z_{2}^{\ga^{n}} \nsim z_{2}^{\ga^{n}}z_{1}$ in $G,$ whereas $\gfi\left(z_{2}^{\ga^{n}}\right) \sim \gfi\left(z_{2}^{\ga^{n}}z_{1}\right)$ in $\gfi(G)$ for every $\Qx$-homomorphism $\gfi$ from $G$ onto a \nilQx \ of finite $\gp$-type. Once this is shown, the theorem will be proved.

First we show that $z_{2}^{\ga^{n}} \nsim z_{2}^{\ga^{n}}z_{1}$ in $G.$ If this were not the case, then there would exist $g \in G$ such that $g^{-1}z_{2}^{\ga^{n}}g = z_{2}^{\ga^{n}}z_{1}.$ Now, since $z_{2} \in \gz_{2}G \setminus Z(G),$ there exists a central element $z$ such that $[z_{2}, \ g] = z.$ Thus,
\[
z_{2}^{\ga^{n}}z_{1} = g^{-1}z_{2}^{\ga^{n}}g = (g^{-1}z_{2}g)^{\ga^{n}} = (z_{2}z)^{\ga^{n}} = z_{2}^{\ga^{n}}z^{\ga^{n}}.
\]
Therefore, $z_{1} = z^{\ga^{n}}.$ This contradicts the fact that $z_{1}$ has no $\ga^{n}$th root in $Z(G).$

Next we show that the images of $z_{2}^{\ga^{n}}$ and $z_{2}^{\ga^{n}}z_{1}$ are conjugate under every $\Qx$-homomorphism from $G$ onto an arbitrary \nilQx \ of finite $\gp$-type. Let $N \unlhd_{\Qx} G$ such that $G/N$ is of finite $\gp$-type and $(z_{1}N)^{\gp^{m}} = N$ for some nonnegative integer $m;$ that is, $z_{1}^{\gp^{m}} \in N.$ Since $\ga$ and $\gp$ are distinct primes, $\ga^{n}$ and $\gp^{m}$ are coprime. Thus, there exist $\gb, \ \gg \in \Qx$ such that $\gb \ga^{n} + \gg \gp^{m} = 1.$ Now, $[z_{2}, \ h] = z_{1}$ and $z_{1} \in Z(G)$ imply that $h^{-\gb}z_{2}h^{\gb} = z_{2}z_{1}^{\gb}$ by Proposition~\ref{p:CommuteCenter}. Therefore,
\begin{eqnarray*}
z_{2}^{\ga^{n}}z_{1} & = & z_{2}^{\ga^{n}}z_{1}^{\gb \ga^{n} + \gg \gp^{m}} = \left(z_{2}z_{1}^{\gb}\right)^{\ga^{n}}z_{1}^{\gg \gp^{m}}\\
                     & = & \left(h^{-\gb}z_{2}h^{\gb}\right)^{\ga^{n}}z_{1}^{\gg \gp^{m}} = h^{-\gb}z_{2}^{\ga^{n}}h^{\gb}\left(z_{1}^{\gp^{m}}\right)^{\gg}.
\end{eqnarray*}
Since $z_{1}^{\gp^{m}} \in N,$ we have $z_{2}^{\ga^{n}}z_{1}N = h^{-\gb}z_{2}^{\ga^{n}}h^{\gb}N$ in $G/N.$ And so, $z_{2}^{\ga^{n}}z_{1}N \sim z_{2}^{\ga^{n}}N$ in $G/N.$ This completes the proof of Theorem~\ref{t:Thm2}.

\vspace{0.25in}

\noindent \underline{Proof of Theorem~\ref{t:Thm3}}: This proof invokes the following lemma:

\begin{lem}\label{l:Lemma3}
Let $G$ be a finitely $\Qx$-generated $\Qx$-torsion-free \nilQx. Suppose that $Y \leq_{\Qx} Z(G),$ and let $H$ be a $\Qx$-isolated subgroup of $G.$ If $I(HY, \ G) = G,$ then $H \unlhd_{\Qx} G.$
\end{lem}

\begin{proof}
The proof is by induction on the $\Qx$-torsion-free rank $r$ of $G$. If $r = 0,$ then $G$ is a $\Qx$-torsion group, contrary to assumption. If $r = 1$, then $G$ is a $\Qx$-module by Corollary~\ref{c:Abelian} and the result is trivial. Suppose then that $r > 1.$ We first prove that $Z(H) \leq_{\Qx} Z(G).$ By Theorem~\ref{t:UniqueRoots}, every element of $G$ has at most one $\ga$th root for every $0 \neq \ga \in \Qx.$ It follows that the centralizer of every non-empty subset of $G$ is $\Qx$-isolated in $G$ (see Lemma 4.8 in \cite{majewicz_and_zyman-2012}). In particular, $Z(H)$ is $\Qx$-isolated in $G$ because $Z(H) = H \cap C_{H}(G),$ where $C_{H}(G)$ denotes the centralizer of $H$ in $G,$ and $H$ is $\Qx$-isolated in $G$ by hypothesis. We will use these observations freely in the remainder of the proof.

Pick any element $g \in G = I(HY, \ G).$ By Theorem~\ref{t:R-Isolator}, there exists $0 \neq \ga \in \Qx$ such that $g^{\ga} \in HY.$ We assert that $g^{\ga}$ centralizes $Z(H).$ Suppose that $h \in Z(H).$ Since $g^{\ga} \in HY \leq_{\Qx} HZ(G),$ there exist elements $k \in H$ and $z \in Z(G)$ such that $g^{\ga} = kz.$ Now, $[h, \ k] = 1$ because $h \in Z(H),$ and so $[h, \ g^{\ga}] = 1$ because $z \in Z(G).$ Thus, $g^{\ga}$ centralizes $Z(H)$ as asserted. Therefore, $g$ also centralizes $Z(H)$ because the centralizer of $Z(H)$ is $\Qx$-isolated in $G.$ Since $g$ is an arbitrary element of $G,$ we have that $Z(H) \leq_{\Qx} Z(G)$ as claimed.

The fact that $Z(H) \leq_{\Qx} Z(G)$ immediately gives $Z(H) \unlhd_{\Qx} G.$ Clearly,  $G/Z(H)$ is finitely $\Qx$-generated. It is also $\Qx$-torsion-free by Lemma~\ref{l:R-IsolatedAndR-TFree}. Since $G$ has $\Qx$-torsion-free rank $r,$ $G/Z(H)$ has $\Qx$-torsion-free rank less than $r$ by Theorem~\ref{t:R-TorsionFreeRank} (i) and Corollary~\ref{c:R-TorsionFreeRank}. We check that the remaining hypotheses of the lemma hold for $G/Z(H).$ Indeed,
$YZ(H)/Z(H) \leq_{\Qx} Z(G/Z(H))$ since $Y \leq_{\Qx} Z(G).$ Moreover, since $H$ is $\Qx$-isolated in $G$, it is easy to see that $H/Z(H)$ must be $\Qx$-isolated in $G/Z(H).$ It remains to show that $G/Z(H) \leq_{\Qx} I(HY/Z(H), \ G/Z(H))$ (the reverse inclusion is trivially true). We begin by noting that $Z(H) \unlhd_{\Qx} HY$ since $Y \leq_{\Qx} Z(G).$ Suppose $gZ(H) \in G/Z(H)$ for some $g \in G.$ Since $G = I(HY, \ G),$ $g^{\gb} \in HY$ for some $0 \neq \gb \in \Qx.$ Thus, $g^{\gb}Z(H) \in HY/Z(H);$ that is, $(gZ(H))^{\gb} \in HY/Z(H).$ And so, $gZ(H) \in I(HY/Z(H), \ G/Z(H))$ as required. By induction, $H/Z(H) \unlhd_{\Qx} G/Z(H),$ and consequently, $H \unlhd_{\Qx} G.$
\end{proof}

We now prove Theorem~\ref{t:Thm3} by induction on the nilpotency class $c$ of $G.$ If $c = 1$, then $G$ is a free $\Qx$-module of finite rank. By Lemma~\ref{l:R-IsolatedAndR-TFree}, $G/H$ is $\Qx$-torsion-free, and thus, a free $\Qx$-module as well. Hence, $G = H \times K$ for some $K \leq_{\Qx} G.$

We claim that
\[
G^{{\gp}^{i}}H = H \times K^{{\gp}^{i}} \hbox{ \, \, \, \, }(i = 1, \ 2, \ \ldots).
\]
To begin with, observe that $HK^{\gp^{i}} = H \times K^{\gp^{i}}$ because $H \cap K^{\gp^{i}} = 1.$ We show that $H K^{\gp^{i}} = G^{\gp^{i}}H.$ Clearly, $H K^{\gp^{i}} \leq_{\Qx} G^{\gp^{i}}H.$ We assert that $G^{\gp^{i}}H \leq_{\Qx} HK^{\gp^{i}}.$ If $g^{\gp^{i}}h$ is a generator of $G^{\gp^{i}}H,$ then
\[
g^{\gp^{i}}h = \left(h_{1}k_{1}\right)^{\gp^{i}}h = h_{1}^{\gp^{i}}hk_{1}^{\gp^{i}}
\]
for some $h_{1} \in H$ and $k_{1} \in K.$ Thus, $g^{\gp^{i}}h \in HK^{\gp^{i}},$ proving the assertion and the claim.

Now, $K$ is finitely $\Qx$-generated by Theorem~\ref{t:MAX}, as well as $\Qx$-torsion-free.
Thus, $\bigcap_{i = 1}^{\infty}K^{{\gp}^{i}} = 1$ by Corollary 5.3 in \cite{majewicz_and_zyman-2012}. And so,
\[
\bigcap_{i = 1}^{\infty} G^{{\gp}^{i}} H = \bigcap_{i = 1}^{\infty}\left(H \times K^{{\gp}^{i}}\right) = H \times \left(\bigcap_{i = 1}^{\infty} K^{{\gp}^{i}}\right) = H.
\]
This proves the theorem when $c = 1.$

Suppose that $c > 1.$ Put
\[
Z = Z(G), \ \ I = I(HZ, \ G), \textmd{\ and \ } L = \bigcap_{i = 1}^{\infty} G^{{\gp}^{i}}H.
\]
We need to show that $L = H.$ By Corollary 3.3 in \cite{majewicz_and_zyman-2012}, $G/Z$ is $\Qx$-torsion-free because $G$ is $\Qx$-torsion-free. Since $I$ is $\Qx$-isolated in $G,$ $I/Z$ is $\Qx$-isolated in $G/Z.$ By induction,
\begin{equation} \label{e:InductionIsolator}
\bigcap_{i = 1}^{\infty} \left(\frac{G}{Z}\right)^{{\gp}^{i}} \cdot \frac{I}{Z} = \frac{I}{Z}.
\end{equation}
We claim that $\bigcap_{i = 1}^{\infty} G^{{\gp}^{i}}I = I.$ Clearly, $\bigcap_{i = 1}^{\infty} G^{{\gp}^{i}}I \geq_{\Qx} I.$ Let $x \in \bigcap_{i = 1}^{\infty} G^{{\gp}^{i}}I.$ Then
\[
xZ \in \frac{G^{\gp^i}I}{Z} = \frac{G^{\gp^i}Z}{Z} \cdot  \frac{I}{Z} = \left(\frac{G}{Z} \right)^{\gp^i} \cdot \frac{I}{Z}
\]
for all $i \geq 1$. It follows from (\ref{e:InductionIsolator}) that $x \in I$, establishing our claim. Now, $\bigcap_{i = 1}^{\infty} G^{{\gp}^{i}}H \leq_{\Qx} I.$ It follows that $H \leq_{\Qx} L \leq_{\Qx} I.$ Now, $I$ is finitely $\Qx$-generated by Theorem~\ref{t:MAX} and $\Qx$-torsion-free, and $H$ is $\Qx$-isolated in $I.$ By Lemma~\ref{l:Lemma3}, $H \unlhd_{\Qx} I$. Moreover, $I/H$ is $\Qx$-torsion-free by Lemma~\ref{l:R-IsolatedAndR-TFree} because $H$ is $\Qx$-isolated in $I.$

We show that $L = H.$ Consider the factor $\Qx$-group $L/H.$ Let $\ell H \in L/H,$ and let $i$ be a positive integer. By Theorem~\ref{t:ProductOfPowers}, every element of $G^{{\gp}^{i}}$ is a $\gp^{i}$th power of an element of $G.$ Thus, there exist $h \in H$ and $g_{i} \in G$ such that
\[
\ell = g_{i}^{{\gp}^{i}}h.
\]
Hence, $g_{i}^{\gp^{i}} = \ell h^{-1} \in I.$ Since $I$ is $\Qx$-isolated in $G,$ $g_{i} \in I.$ This implies that $\ell H = (g_{i}H)^{\gp^{i}}$ with $g_{i}H \in I/H.$ It follows that $\ell H$ has a $\gp^{i}$th root in $I/H$ for every $i = 1, \ 2, \ \ldots.$ If $\ell H \neq H$, this would contradict Lemma \ref{l:NoInfiniteRoots}. Consequently, $\ell H = H,$ and thus, $L = H$ as desired. This finishes the proof of Theorem~\ref{t:Thm3}.

\vspace{.25in}

\noindent \underline{Proof of Theorem~\ref{t:Thm4}}: Let $c$ be the nilpotency class of $G.$ We can assume, without loss of generality, that $G$ is $\Qx$-torsion-free. To see why, put $T = \gt(G)$ and fix a prime $\gp \in \Qx.$ Since $T$ is a (normal) $\Qx$-subgroup of $G,$ it is finitely $\Qx$-generated by Theorem~\ref{t:MAX}, and thus, of finite type. Let $\gs \in \Qx$ be the exponent of $T,$ and assume that $(gT)^{\ga} \in HT/T$ for some $0 \neq \ga \in \Qx.$ There exist $h \in H$ and $t \in T$ such that $g^{\ga} = ht.$ 
We claim that there exist non-zero distinct elements $\gb$ and $\gm$ of $\Qx$ and $t' \in T$ such that $g^{\ga \gb} = h^{\gb}t'$ and $g^{\ga \gm} = h^{\gm}t'.$
Choose any $0 \neq \gb \in \Qx,$ and let $\gm = \gb + \gs.$ Put $\gt_{i} = \gt_{i}(h, \ t)$ for $i \in \N.$ Since $T \unlhd_{\Qx} G,$ $\gt_{i} \in T$ for $i \geq 2$ (see the proof of Theorem 4.16 in \cite{clement-majewicz-zyman} for details). By the Hall-Petresco Axiom,
\begin{eqnarray*}
g^{\ga \gb} & = & h^{\gb}t^{\gb}\gt_{c}^{- \binom{\gb}{c}} \cdots \gt_{2}^{- \binom{\gb}{2}}\\
            & = & h^{\gb}t'
\end{eqnarray*}
with $t' \in T$ and
\begin{eqnarray*}
g^{\ga \gm} & = & h^{\gm}t^{\gm}\gt_{c}^{- \binom{\gm}{c}} \cdots \gt_{2}^{- \binom{\gm}{2}}\\
            & = & h^{\gm}t^{\gb + \gs}\gt_{c}^{- \binom{\gb + \gs}{c}} \cdots \gt_{2}^{- \binom{\gb + \gs}{2}}.
\end{eqnarray*}
Since $T$ has exponent $\gs,$ $t^{\gb + \gs} = t^{\gb}$ and $\gt_{i}^{- \binom{\gb + \gs}{i}} = \gt_{i}^{- \binom{\gb}{i}}$ for $2 \leq i \leq c.$ This proves the claim.

Now, $\ga \gb - \ga \gm \neq 0$ and $g^{\ga \gb - \ga \gm} = h^{\gb - \gm} \in H,$ contrary to the hypothesis. This means that $(gT)^{\ga} \notin HT/T$ for any $0 \neq \ga \in \Qx.$ {Hence, $g^{\ga} \notin HT,$ and thus, $g \notin T.$ Consequently, we might as well assume that $G$ is $\Qx$-torsion-free as mentioned in the beginning of the proof.}

Since $I(H, \ G)$ is a $\Qx$-isolated subgroup of $G,$
\[
\bigcap_{i = 1}^{\infty} G^{{\gp}^{i}}I(H, \ G) = I(H, \ G)
\]
by Theorem~\ref{t:Thm3}. Now,
\[
g \notin \bigcap_{i = 1}^{\infty} G^{{\gp}^{i}}I(H, \ G)
\]
because $g \notin I(H, \ G)$ by hypothesis. Furthermore,
\[
\bigcap_{i = 1}^{\infty} G^{{\gp}^{i}}H \subseteq \bigcap_{i = 1}^{\infty} G^{{\gp}^{i}}I(H, \ G)
\]
since $H \leq_{\Qx} I(H, \ G).$ And so, there exists an integer $j$ such that $G/G^{\gp^{j}}$ is a \nilQx \ of finite $\gp$-type separating $g$ and $H.$

\end{document}